\documentclass{article}

\usepackage{amsthm, amssymb, amsmath}
\usepackage{verbatim}
\usepackage{graphics}      
\usepackage{graphicx}      
\usepackage{setspace}

\newtheorem{thm}{Theorem}
\newtheorem{prop}[thm]{Proposition}
\newtheorem{cor}[thm]{Corollary}
\newtheorem{lem}[thm]{Lemma}
\newtheorem{conj}{Conjecture}
\newtheorem*{rmk}{Remark}

\theoremstyle{plain}
\newtheorem*{defn}{Definition}

\newcommand{\R}{\mathbb{R}}

\newcommand{\etc}{\textrm{etc.}}
\newcommand{\rad}{\mathrm{Rad}}

\title{On the Number of ABC Solutions with Restricted Radical Sizes}
\author{Daniel M. Kane}

\begin{document}
\maketitle

\begin{abstract}
We establish upper and lower bounds for the number of solutions to $A+B=C$ in relatively prime integers $A, B, C$ with
$\max(|A|, |B|, |C|) \leq N$  and $\rad(A) \leq N^a, \rad(B) \leq N^b, \rad(C) \leq N^c$,
valid when $0 < a, b, c \leq 1$. The lower bound is $\Omega(N^{a+b+c -1}(\log N)^{-2}),$
and the upper bound is of the form $O( N^{a+b+c-1+ \epsilon} + N^{1+\epsilon})$,  for any fixed $\epsilon>0$. In particular, these bounds match up to $N^\epsilon$ factors so long as $a+b+c\geq 2$, verifying a conjecture or Mazur in this parameter range.
\end{abstract}

\section{Introduction}

The abc-conjecture of Masser and Oesterl\'{e} is a famous unifying conjecture in number theory. For an integer $n$, let the radical of $n$ be given by $\rad(n) = \prod_{p|n} p$ be the product of its prime divisors.  The abc-Conjecture states roughly that the equation $A+B+C=0$ has no solutions in highly divisible, relatively prime integers $A, B, C$.In its strong form the conjecture can be  stated as follows (\cite[Conjecture 12.2.2]{BG}).

\begin{conj}\label{ABCConj} {\em (abc-Conjecture in the strong form)}
For any $\epsilon >0$ there are only finitely many solutions to the equation
$A+B+C=0$ in relatively prime integers $A,B,C$ so that
$\max(|A|,|B|,|C|) > \rad(ABC)^{1+\epsilon}.$
\end{conj}
\noindent For various equivalent versions of the abc-Conjecture and its important consequences
in number theory see  Bombieri and Gubler \cite[Chapter 12]{BG}.

Although the abc-Conjecture is not known to hold for any value of $\epsilon$, when the restrictions on $A, B,$ and $C$ are loosened the conjectural predictions on the number of solutions become more tractable. In particular, Mazur \cite{mazur} put forth
the following conjecture (stated in slightly different terminology), counting solutions measured
in terms of the sizes of their radicals.

\begin{conj}\label{mainConj}
Given constants $0<a,b,c\leq 1$, and $\epsilon>0$, let $S_{a,b,c}(N)$ be the number of triples of relatively prime integers, $A,B,C$ with $A+B+C=0$, $|A|,|B|,|C|\leq N$ and $\rad(A)\leq |A|^a, \rad(B)\leq |B|^b, \rad(C)\leq |C|^c$.

Then for any fixed $a,b,c$ with $a+b+c>1$ and $\epsilon >0$, then for all sufficiently large $N$
$$
N^{a+b+c-1-\epsilon} < S_{a,b,c}(N) < N^{a+b+c-1+\epsilon}.
$$
\end{conj}

We will henceforth refer to $S_{a,b,c}(N)$ as the number of solutions to the abc problem with parameters $(a,b,c)$ or with parameters $(a,b,c,N)$.

In \cite{mazur}, Mazur alludes to a proof of this Conjecture in the case when $5/6\leq a,b,c\leq 1$.  We extend this result to a wider range of values of $a,b$ and $c$.

In terms of lower bounds, we show that the lower bound in Conjecture \ref{mainConj} holds for all $a+b+c>1$. In particular, a lower bound was obtained in an unpublished note of Granville \cite{Gran}. We prove a slight strengthening of this result and in particular prove that:

\begin{thm}\label{lowerBoundThm}
For $0<a,b,c\leq 1,$ with $a+b+c>1$, the number of solutions to the abc problem with parameters $(a,b,c,N)$ is $\Omega(N^{a+b+c-1}\log(N)^{-2})$.
\end{thm}

The upper bounds prove to be somewhat more difficult. There, we manage to show the bounds proposed in Conjecture \ref{mainConj} when $a+b+c\geq 2$. Note that below and throughout the rest of the paper, when using $N^\epsilon$ in asymptotic notation, it will be taken to mean that the bound holds for any $\epsilon>0$ but that the implied constants may depend on $\epsilon$.

\begin{thm}\label{upperBoundThm}
For $0<a,b,c\leq 1,$ the number of solutions to the abc problem with parameters $(a,b,c,N)$ is $O(N^{a+b+c-1+\epsilon}+N^{1+\epsilon})$.
\end{thm}

We use two main techniques to prove these Theorems.  First we use lattice methods.  The idea is to fix the non-squarefree parts of $A,B,C$ and to then count the number of solutions.  For example if $A,B,C$ have non-squarefree parts $\alpha,\beta,\gamma$, we need to count solutions to an equation of the form $\alpha X+\beta Y+\gamma Z=0$.

Secondly, we use a result of Heath-Brown on the number of integer points on a conic.  The idea here is to note that if $A,B,C$ are highly divisible, they are likely divisible by large squares.  Writing $A=\alpha X^2, B=\beta Y^2, C=\gamma Z^2$, then for fixed values of $\alpha,\beta,\gamma$ the solutions to the abc problem correspond to integer points of small size on a particular conic.

In Section 2, we cover the lattice methods.  In particular, in Section \ref{latticeLowerBoundsSec}, we use these methods to prove Theorem \ref{lowerBoundThm}. The proof of our upper bound will involve breaking our argument into cases based upon the approximate sizes of various parameters of $A$, $B$, and $C$.  In Section \ref{DaidicIntervalsSec}, we introduce some ideas and notation that we will use when making these arguments.  In Section \ref{LatticeUpperBoundsSec}, we use lattice methods to prove an upper bound on the number of solutions.  These techniques will work best when $A$, $B$, and $C$ have relatively few repeated factors.  In Section \ref{conicsSec}, we prove another upper bound, this time using our methods involving conics.  These results will turn out to be most effective when $A$,$B$ and $C$ have many repeated factors. Finally, we combine these results with our upper bound from the previous Section to prove Theorem \ref{upperBoundThm}.

\section{Lattice Methods}

\subsection{Lattice Lower Bounds}\label{latticeLowerBoundsSec}

In this section we will prove Theorem \ref{lowerBoundThm} using lattice methods. In particular, we will show that many solutions can be found by looking for points of small size in a well-chosen lattice. For a basic references on lattices see the first couple chapters of \cite{lattice}. In order for this to work, we will need bounds on the number of lattice points in certain regions. For this, we will need the following Lemma:

\begin{lem}\label{latticeBoundsLem}
Let $L$ be a lattice in a two dimensional vector space $V$ and $P$ a convex polygon in $V$.  Let $m$ be the minimum separation between points of $L$.  Then
$$
|L\cap P| = \frac{\textrm{Volume}(P)}{\textrm{CoVolume}(L)} + O\left(\frac{\textrm{Perimeter}(P)}{m} + 1 \right).
$$
\end{lem}
\begin{proof}
Begin with a reduced basis of $L$.  We apply a linear transformation to $V$ so that $L$ becomes a square lattice by fixing the short vector in the basis and sending the other vector to an orthogonal vector of length $m$.  This operation has no effect on $|L\cap P|$ or $\frac{\textrm{Volume}(P)}{\textrm{CoVolume}(L)}$ and can only increase $\frac{\textrm{Perimeter}(P)}{m}$ by at most a constant factor, therefore, it suffices to consider our problem in the case where $L$ is a square lattice.  We now note that if we draw a fundamental parallelepiped around each point of $|L\cap P|$, their union is sandwiched between the set of points within distance $\sqrt{2}m$ of $P$, and the set of points where the disc of radius $\sqrt{2}m$ around them is contained in $P$.  Thus, $m^2$ times the number of such points is bounded between the areas of these two regions, which gives our result.
\end{proof}
\begin{rmk}
Note that we will need to apply Lemma \ref{latticeBoundsLem} in the case when $L$ is a $2$-dimensional sublattice of $\R^3$. In this case, we will set $V$ to be the real span of $L$ and define the covolume of $L$ and volume of $P$ using the measure on $V$ coming from the induced metric.
\end{rmk}

The basic idea of our proof of Theorem \ref{lowerBoundThm} will be as follows.  We begin by picking relatively prime integers $\alpha,\beta,\gamma$ so that $\alpha/\rad(\alpha) > N^{1-a}$, $\beta/\rad(\beta) > N^{1-\beta}$ and $\gamma/\rad(\gamma) > N^{1-c}$ (for $a,b,c$ as given by our abc problem).  We will then look for solutions to the abc problem in which $\alpha|A, \beta|B$ and $\gamma|C$.  We note that since $\alpha,\beta,\gamma$ each have such small radicals, that any such $A,B,C$ each less than $N$ will automatically satisfy $\rad(A)<|A|^a, \rad(B)<|B|^b$ and $\rad(C)<|C|^c$.  We are left with the problem of finding such $A,B,C$ that are relatively prime and sum to 0.  The set of such $A,B,C$ with sum 0 form a 2-dimensional lattice, and the set with $|A|,|B|,|C|\leq N$ a convex polygon. Hence the number of such solutions may be counted using Lemma \ref{latticeBoundsLem}.  We may additionally find the number of such triples with $A,B,C$ relatively prime by using sieve methods.  This will suffice to provide an appropriate lower bound, unless the shortest vector of the lattice in question is very small.  In the following Lemma, we show that we can pick $\alpha,\beta,\gamma$ to avoid such problems.

\begin{lem}\label{evenLatticeLem}
Let $0<a\leq b \leq c \leq 1$ so that $a+b+c>1$.  Let $\delta = a+b+c-1$.  Let $N$ be a positive integer.  Then there exists a prime $5<q=O(\log N)$ and integers $\alpha = 2^x q^w$, $\beta = 3^y$, $\gamma = 5^z$ (for $w,x,y,z$ positive integers) so that:
\begin{itemize}
\item $2qN^{1-a} \geq \alpha/\rad(\alpha) \geq N^{1-a}$
\item $3N^{1-b} \geq \beta/\rad(\beta) \geq N^{1-b}$
\item $5N^{1-c} \geq \gamma/\rad(\gamma) \geq N^{1-c}$
\item For any non-zero integers $A,B,C$ with $A+B+C=0$ and $\alpha|A,\beta|B,\gamma|C$ we have that $\max(|A|,|B|,|C|)\geq \Omega(N^{1-\delta/2}/\log(N))$
\end{itemize}
\end{lem}
The basic idea of the proof will be to begin with $\alpha = 2^x$, $\beta = 3^y$, $\gamma=5^z$ for appropriate values of $x,y,z$.  This will work unless the associated lattice, $L_0$, has a particularly small shortest vector.  In this case, we replace some of the factors of 2 in $\alpha$ by factors of $q$.  The new lattice, $L$, will have a reasonably short vector given by an appropriate multiple of the old shortest vector.  Our result will follow from noting that this will be a relatively short vector that is not a multiple of another vector in the lattice.

\begin{proof}
We let $y$ and $z$ be the smallest integers so that $3^{y-1}\geq N^{1-b}$ and $5^{z-1}\geq N^{1-c}$, and let $\beta=3^y$ and $\gamma=5^z$.  Note that $\beta$ and $\gamma$ clearly satisfy the necessary conditions.  Let $x_0$ be the smallest integer so that $2^{x_0-1}\geq N^{1-a}$, and let $\alpha_0 = 2^{x_0}$.  Consider the lattice $L_0$ of triples of integers $(A,B,C)$ with $A+B+C=0$ and $\alpha_0|A,\beta|B$ and $\gamma|C$.  Let the smallest non-zero vector in this lattice be $(\alpha_0 t,\beta u,\gamma v)$. Note that this lattice has index $\alpha_0\beta\gamma$ within the lattice of all integer triples $A,B,C$ with $A+B+C=0$. Therefore, it has covolume $\Theta(\alpha_0\beta\gamma) = \Theta(N^{2-\delta})$ inside of the plane defined by $A+B+C=0$. Let $m$ be the length of the shortest vector in this lattice. If $m> N^{1-\delta/2}/\log(N)$, we may use $q=7$ and $\alpha=7\alpha_0$.  Otherwise, we may assume that $m \leq N^{1-\delta/2}/\log(N)$.

Let $q$ be the smallest prime not dividing $30t$.  Since $30t =O(N)$, we have that $q= O(\log N)$.  Let $2^h||u$. Let $k$ be the largest integer so that $2^k < 2^h N^{1-\delta/2}/m$.  Note that $2^h N^{1-\delta/2} > 2^k \beta |u| \geq 2^{k+h} N^{1-b}$.  Hence $2^{k} \leq N^{(1-a)/2 - (1-b)/2 + (1-c)/2}$.  $2^h \leq |u| \leq m/\beta \leq N^{(1-a)/2 - (1-b)/2 + (1-c)/2}$.  Thus, $2^{h+k} \leq N^{(1-a)-(1-b)+(1-c)} \leq N^{1-a} \leq 2^{x_0}.$

Note that therefore $2^h|\alpha_0 t , \beta u$. Since $\alpha_0 t + \beta u + \gamma v = 0$ and since $(\gamma,2)=1$, this implies that $2^h | v$.

Let $x = x_0 - h - k$.  Let $w$ be the smallest positive integer so that $q^{w-1}2^{x-1} \geq N^{1-a}$.  Let $\alpha = 2^x q^w$.  Clearly, $2qN^{1-a} \geq \alpha/\rad(\alpha) \geq N^{1-a}$.  Let $t'=2^k t$, $u' = 2^{-h}q^w u$, $v' = 2^{-h}q^w v$ (which are all integers by the above).  Notice that $\alpha t' + \beta u' + \gamma v' = 0.$  Furthermore, note that $m' := \max(|\alpha t'|,|\beta u'|,|\gamma v'|) = 2^{-h}q^w m.$  Now, $2^x = \Theta(N^{1-a} 2^{-h-k})$.  Therefore $q^2 2^{h+k} \gg q^w \gg q2^{h+k}$.  Thus, $q^2 N^{1-\delta/2} \gg m' \gg qN^{1-\delta/2}$.

We now consider the lattice, $L$ of triples $(A,B,C)$ with $A+B+C=0$ and $\alpha|A,\beta|B,\gamma|C$.  We wish to show that the shortest non-zero vector in this lattice has length at least $\Omega(N^{1-\delta/2}/\log(N))$.  We split into cases based upon whether or not $(\alpha t', \beta u',\gamma v')$ is a multiple of this shortest vector.

If $(\alpha t', \beta u', \gamma v')$ is a multiple of this shortest vector, we claim that it must be this shortest vector (up to sign).  This is because $\gcd(t',u',v')=1$.  This is true because $\gcd(t,u,v)=1$, $2\not | u'$, $q\not | t'$.  Hence the shortest vector in $L$ must have length $m'$ and we are done.

If $(\alpha t', \beta u', \gamma v')$ is not a multiple of the shortest vector, we use the fact that the product of the length of the shortest vector of a 2-dimensional lattice with the length of any non-multiple of the shortest vector is at least some constant multiple of the covolume.  Since the covolume of $L$ is $\Omega(\alpha\beta\gamma)=\Omega(qN^{2-\delta})$, the length of the shortest non-zero vector is at least $\Omega(qN^{2-\delta})/m' = \Omega(N^{1-\delta/2}/\log(N)).$
\end{proof}

We are now ready to prove Theorem \ref{lowerBoundThm}.

\begin{proof}[Proof of Theorem \ref{lowerBoundThm}]
Assume without loss of generality that $a\leq b\leq c$. Let $\delta = a+b+c-1 > 0$.  Let $q,\alpha,\beta,\gamma$ be as given in Lemma \ref{evenLatticeLem}.

Let $L$ be the lattice of triples of integers $A,B,C$ so that $A+B+C=0$ and $\alpha|A,\beta|B,\gamma|C$.  Let $L_n$ be the sublattice of $L$ consisting of the triples $(A,B,C)$ so that $n|A,B,C$.  Note that with an appropriate normalization of the area on the plane $A+B+C=0$ that $L$ has covolume $U=\alpha\beta\gamma$, and that $\log^2(N) N^{2-\delta} \gg U \gg N^{2-\delta}$.  Note that $L_n$ has covolume $\frac{n^2U}{\gcd(n,30q)}$.  Let $M$ be the length of the shortest non-zero vector in $L$, and recall that $M \gg N^{1-\delta/2}\log^{-1} (N)$.

Let $P$ be the polygon in the plane $A+B+C=0$ defined by $|A|,|B|,|C|\leq N$.  The number of solutions to the abc problem with parameters $(a,b,c,N)$ is at least the number of points in $L\cap P$ with relatively prime coordinates.  This is
$$
\sum_n \mu(n) |L_n \cap P| = \sum_{n=1}^{O(N \log(N) /M)} \mu(n) |L_n \cap P|.
$$
We can cut off the sum because for $n$ squarefree, a vector $v$ is in $L_n$ only if $\frac{\gcd(n,30q)v}{n}$ is in $L$.  This can happen only if $|v| \geq n M / (30 q)$.  Hence the summand is trivial for all $n$ bigger than a sufficiently large multiple of $qN/M$.

Letting $V$ be the volume of $P$, and noting that the shortest vector in $L_n$ has length $\Omega(nM/q)$, by Lemma \ref{latticeBoundsLem} the above equals
$$
\sum_{n=1}^{O(N\log (N)/M)} \left(\frac{\mu(n)\gcd(n,30q)}{n^2}\right)\left(\frac{V}{U} \right) + O(N\log(N)/(nM)+1).
$$
The main term is
\begin{align*}
\left(\frac{V}{U} \right)\left(\sum_{n=1}^{\infty} \frac{\mu(n)\gcd(n,30q)}{n^2} + O(M/N)\right) & = \Theta\left( \frac{V}{U}\right)\\ & = \Omega\left(\frac{N^2}{\log^2(N)N^{3-a-b-c}}\right) \\ & = \Omega(N^{a+b+c-1}\log^{-2}(N)).
\end{align*}
The error term is
\begin{align*}
O(\log^2(N) N/M + N\log(N)/M) & = O(\log^3(N) N^{\delta/2}) \\
& = O(\log^3(N) N^{(a+b+c-1)/2}).
\end{align*}
This completes our proof.
\end{proof}

\textbf{Note:} This Theorem can be obtained more simply and with better bounds ($\Omega(N^{a+b+c-1})$) in the case when $\min(a,b,c)+\max(a,b,c)>1$.  In this case, simply setting $\alpha = 2^x$,$\beta = 3^y$, $\gamma=5^z$, we note that $L$ contains the non-parallel vectors $(\alpha\beta,-\beta\alpha,0),(\alpha\gamma,0,-\alpha\gamma),(0,\beta\gamma,-\beta\gamma)$, at least two of which have length significantly less than $N$.  This implies an upper bound on the length of the longer vector of a reduced basis of $L$, and thus a lower bound on the length of the shortest non-zero vector.  This bound turns out to be sufficiently to prove our lower bound.

\subsection{Dyadic Intervals}\label{DaidicIntervalsSec}

Before beginning our work on upper bounds, we discuss some ideas involving dyadic intervals that we will make use of.  First a definition:
\begin{defn}
A dyadic interval is an interval of the form $[2^n,2^{n+1}]$ for some integer $n$.
\end{defn}
In the process of proving upper bounds we will often wish to count the number of solutions to an abc problem in which some functions of $A,B,C$ lie in fixed dyadic intervals.  We may for example claim that the number of solutions to an abc problem where $f(A),f(B),f(C)$ lie in fixed dyadic intervals is $O(X)$.  Here $f$ will be some specified function and we are claiming that for any triple of dyadic intervals $I_A,I_B,I_C$ the number of $A,B,C$ that are solutions to the appropriate abc problem and so that additionally $f(A)\in I_A, f(B)\in I_B$ and $f(C)\in I_C$ is $O(X)$.  When we do this, it will often be the case that $X$ depends on $f(A),f(B),f(C)$ and not just the parameters of the original abc problem we were trying to solve.  By this we mean that our upper bound is valid if the $f(A),f(B),f(C)$ appearing in it are replaced by any numbers in the appropriate dyadic intervals.  Generally this freedom will not matter since fixing dyadic intervals for $f(A),f(B),f(C)$ already fixes their values up to a multiplicative constant.  It should also be noted that $[1,N]$ can be covered by $O(\log N)$ dyadic intervals.  Thus if we prove bounds on the number of solutions in which a finite number of parameters (each at most $N$) lie in fixed dyadic intervals, we obtain an upper bound for the number of solutions with no such restrictions that is at most $N^\epsilon$ larger than the bound for the worst set of intervals.

\subsection{Lattice Upper Bounds}\label{LatticeUpperBoundsSec}

In order to prove upper bounds, we will need a slightly different form of Lemma \ref{latticeBoundsLem}.

\begin{lem}\label{latticeUpperBoundsLem}
Let $L$ be a 2 dimensional lattice and $P$ a convex polygon, centrally symmetric about the origin.  Then the number of vectors in $L\cap P$ which are not positive integer multiples of other vectors in $L$ is
$$
O\left(\frac{\textrm{Volume}(P)}{\textrm{CoVolume}(L)} + 1 \right).
$$
\end{lem}
\begin{proof}
If $L\cap P$ only contains the origin or multiples of a single vector, the result follows trivially.  Otherwise $P$ contains two linearly independent vectors of $L$. This means that $P$ contains at least half of some fundamental domain.  Therefore $2P$ contains some whole fundamental domain. Therefore $4P$ contains all of the fundamental domains centered at any of the points in $L\cap P$.  Therefore $|L\cap P| = O\left(\frac{\textrm{Volume}(P)}{\textrm{CoVolume}(L)}  \right).$
\end{proof}

We will also make extensive use of the following proposition:

\begin{prop}\label{radProp}
For any $m$, the number of $k$ with $|k|\leq N$ and $\rad(k)=m$ is $O(N^\epsilon)$, where the implied constant depends on $\epsilon$ but not $N$ or $m$.
\end{prop}
\begin{proof}
Let $m=p_1p_2\cdots p_n$, where $p_1<p_2<\cdots<p_n$ are primes (if $m$ is not squarefree we have no solutions).  Then all such $k$ must be of the form $\prod_{i=1}^n p_i^{a_i}$ for some integers $a_i\geq 1$ with $\sum_i a_i\log(p_i) \leq \log(N)$.  Note that if for each such $k$ you consider the unit cube defined by $\prod_{i=1}^n [a_{i}-1,a_i]\subset \R^n$, these cubes have disjoint interiors and are contained in a simplex of volume $\frac{1}{n!}\prod_i \frac{\log(N)}{\log(p_i)}$.  Hence the number of such $k$ is at most
$$
\frac{1}{n!}\prod_i \frac{\log(N)}{\log(p_i)} = O\left(\frac{\log(N)}{n}\right)^n.
$$
Now we must also have that $n!\leq \prod_i p_i\leq N$ or there will be no solutions, so $n=O\left(\frac{\log(N)}{\log\log(N)}\right)$.  Now $n\log N - n\log n$ is increasing for $n<N/e$ so the number of solutions is at most
\begin{align*}
O\left(\frac{\log(N)}{\log(N)/\log\log(N)}\right)^{O\left(\frac{\log(N)}{\log\log(N)}\right)} & = \exp\left(O\left( \frac{\log(N)\log\log\log(N)}{\log\log(N)}\right)\right) \\ & = O(N^\epsilon).
\end{align*}
\end{proof}

We now need some more definitions.  For an integer $n$ define
$$
u(n) := \prod_{p||n} p
$$
to be the product of primes that divide $n$ exactly once.  Let
$$
e(n) := \prod_{p^\alpha||n, \alpha>1} p^\alpha = n/u(n)
$$
be the product of primes dividing $n$ more than once counted with their appropriate multiplicity.  Finally, let
$$
v(n) = \prod_{p^2|n} p = \rad(n)/u(n) = \rad(e(n))
$$
be the product of primes dividing $n$ more than once.

We can now prove the first part of our upper bound:

\begin{prop}\label{latticeUpperBoundProp}
Fix $0<a,b,c\leq 1$.  The number of solutions to the abc problem with parameters $(a,b,c,N)$ and with $v(A),v(B),v(C)$ lying in fixed dyadic intervals is
$$
O(N^{a+b+c-1+\epsilon} + v(A)v(B)v(C)N^\epsilon).
$$
\end{prop}
Note that since $v(A),v(B),v(C)\leq \sqrt{N}$, this would already give a weaker but non-trivial version of Theorem \ref{upperBoundThm}.
\begin{proof}
We will begin by additionally fixing dyadic intervals for $|A|, |B|, |C|$, $u(A),$ $u(B), u(C)$, $e(A), e(B), e(C)$.  Since there are only $\log(N)$ possible intervals for each, our total number of solutions will be greater by a factor of at most a factor of $O(N^\epsilon)$.  Assume without loss of generality that $|A|\leq |B| \leq |C|$.

There are $O(v(A)v(B)v(C))$ ways to fix the values of $v(A),v(B),v(C)$ within their respective dyadic intervals.  Given these, by Proposition \ref{radProp} there are $O(N^\epsilon)$ possible values of $e(A),e(B),e(C)$.  Pick a triple of values for $e(A),e(B),e(C)$.  We assume these are relatively prime, for otherwise they could not correspond to any valid solutions to our abc problem.  Define the lattice $L$ to consist of triples of integers which sum to 0, and are divisible by $e(A),e(B),e(C)$ respectively.  We define the polygon $P$ to be the set of $(x_1,x_2,x_3)$ so that $x_1+x_2+x_3=0$ and $|x_1|$ is bounded by the upper end of the dyadic interval for $|A|$, and $|x_2|,|x_3|$ are likewise bounded by the intervals for $|B|$ and $|C|$.  The number of solutions to our abc problem with the specified values of $e(A),e(B),e(C)$ and with $|A|,|B|,|C|$ in the appropriate dyadic intervals is at most the number of vectors in $L\cap P$ that are not positive integer multiples of other vectors in $L$.  By Lemma \ref{latticeUpperBoundsLem} this is
$$
O\left(\frac{\textrm{Volume}(P)}{\textrm{CoVolume}(L)} + 1 \right) = O\left( \frac{|AB|}{e(A)e(B)e(C)} + 1\right).
$$
Multiplying this by the number of ways we had to choose values for $v(A),$ $v(B),$ $v(C),$ $e(A),e(B),e(C)$, we get that the total number of solutions to our original abc problem with the specified dyadic intervals for $|A|,v(A),e(A),$ $\etc$ is at most
\begin{align*}
O & \left(\frac{N^\epsilon |ABC|v(A)v(B)v(C)}{e(A)e(B)e(C) |C|} + N^\epsilon v(A)v(B)v(C)\right)\\
& =  O\left( \frac{N^\epsilon \rad(A)\rad(B)\rad(C)}{|C|} + N^\epsilon v(A)v(B)v(C)\right)\\
& \leq  O(N^\epsilon |A|^a|B|^b|C|^{c-1} + N^\epsilon v(A)v(B)v(C) ).
\end{align*}
Where the last step comes from noting that for any solution to our abc problem, $\rad(A)\leq |A|^a$, etc. Since $a+b+c>1$, $|A^aB^bC^{c-1}|$ is maximized with respect to $N\geq |C|\geq |B| \geq |A|$ when $|A|=|B|=|C|=N$.  So we obtain the bound
$$
O(N^{a+b+c-1+\epsilon} + N^\epsilon v(A)v(B)v(C)).
$$
\end{proof}

\section{Points on Conics}\label{conicsSec}

The bound from Proposition \ref{latticeUpperBoundProp} is useful so long as $v(A),v(B),v(C)$ are not too big.  When they are large, we shall use different techniques.  In particular, if $v(A),v(B),v(C)$ are large, then $A,B,C$ are divisible by large squares.  We will fix non-square parts of these numbers and bound the number of solutions using a Theorem from \cite{HB}:

\begin{thm}[\cite{HB} Theorem 2]
Let $q$ be an integral ternary quadratic form with matrix $M$. Let
$\Delta = | \det M|$, and assume that $\Delta \neq 0$. Write $\Delta_0$ for the highest common factor
of the $2 \times 2$ minors of $M$. Then the number of primitive integer solutions of
$q(x) = 0$ in the box $|x_i| \leq Ri$ is
$$
\ll \{1+ (\frac{R_1R_2R_3\Delta_0^2}{\Delta})^{1/2}\}d_3(\Delta).
$$
Where $d_3(\Delta)$ is the number of ways of writing $\Delta$ as a product of three integers.
\end{thm}

Putting this into a form that fits our needs slightly better:
\begin{cor}\label{conicCor}
For $a,b,c$ relatively prime integers, the number of solutions to $aX^2+bY^2+cZ^2=0$ in relatively prime integers $X,Y,Z$ with $|X|\leq R_1$, $|Y|\leq R_2$, $|Z|\leq R_3$ is
$$
O\left(\left(1+\sqrt{\frac{R_1R_2R_3}{|abc|}} \right)(|abc|)^\epsilon \right).
$$
\end{cor}
\begin{proof}
This follows immediately by applying the above Theorem to the obvious quadratic form, noting that $\Delta=|abc|$, $\Delta_0=1$ and that $d_3(N)=O(N^\epsilon)$.
\end{proof}

We now have all of the machinery ready to prove the upper bound.  We make one final pair of definitions.

Let
$$
S(n) := \prod_{p^\alpha || n} p^{\lfloor\alpha/2\rfloor} = \sup \{ m: m^2|n\}.
$$
be the largest number whose square divides $n$.
Let
$$
T(n) = \left|\frac{n}{ S(n)^2}\right|.
$$
In particular, $|n|=T(n)S(n)^2$ for all $n$ with $T(n)$ squarefree.

We use Corollary \ref{conicCor} to prove another upper bound.
\begin{prop}\label{conicUpperBoundProp}
Fix dyadic intervals for $S(A),T(A),S(B),T(B),S(C),$ and $T(C)$.  The number of solutions of the abc problem with $S$ and $T$ of $A,B,C$ lying in these intervals is
$$
O\left(\left(T(A)T(B)T(C)+\sqrt{{S(A)S(B)S(C)}{T(A)T(B)T(C)}} \right)N^\epsilon \right).
$$
\end{prop}
\begin{proof}
There are $O(T(A)T(B)T(C))$ choices for the values of $T(A),T(B),$ $T(C)$ lying in their appropriate intervals.  Fixing these values, we count the number of solutions to
$$
\pm T(A)S(A)^2 \pm T(B)S(B)^2 \pm T(C)S(C)^2 = 0
$$
with $S(A),S(B),S(C)$ relatively prime and in the appropriate intervals.  By Corollary \ref{conicCor} this is at most
$$
O\left(\left(1+\sqrt{\frac{S(A)S(B)S(C)}{T(A)T(B)T(C)}} \right)N^\epsilon \right).
$$
Hence the total number of solutions to our abc problem with $T(A),S(A),$ $\etc$ lying in appropriate intervals is
$$
O\left(\left(T(A)T(B)T(C)+\sqrt{{S(A)S(B)S(C)}{T(A)T(B)T(C)}} \right)N^\epsilon \right).
$$
\end{proof}

We are now prepared to prove our upper bound on the number of solutions to an abc problem.

\begin{proof}[Proof of Theorem \ref{upperBoundThm}]
It is enough to prove our Theorem after fixing $S(A),$ $T(A),$ $v(A)$, $S(B),T(B),v(B)$, $S(C),T(C),v(C)$ to all lie in fixed dyadic intervals, since there are only $O(\log(N)^9)=O(N^\epsilon)$ choices of these intervals.  It should be noted that $S(n)\geq v(n)$ for all $n$.  By Proposition \ref{latticeUpperBoundProp} we have the number of solutions is at most
$$
O\left(N^{a+b+c-1+\epsilon} + S(A)S(B)S(C) N^\epsilon\right).
$$
By Proposition \ref{conicUpperBoundProp}, noting that $N\geq T(A)S(A)^2,T(B)S(B)^2,T(C)S(C)^2$, we know that the number of solutions is at most
$$
O\left( N^{3+\epsilon} (S(A)S(B)S(C))^{-2} + N^{3/2+\epsilon}(S(A)S(B)S(C))^{-1/2}\right).
$$
If $S(A)S(B)S(C)\geq N$ this latter bound is $O(N^{1+\epsilon})$ and if $S(A)S(B)S(C)\leq N$, the former bound is $O(N^{a+b+c-1+\epsilon} + N^{1+\epsilon})$.  So in either case we have our desired bound.
\end{proof}

\section{Conclusion}

We have proved several bounds on the number of solutions to abc type problems.  In particular we have proven the Conjecture \ref{mainConj} so long as $a+b+c\geq 2$.

Our upper bounds likely cannot be extended much, because when $S(A)S(B)S(C)\sim N$, then Corollary \ref{conicCor} only says that we have $O(N^\epsilon)$ solutions for each choice of the $T$'s.  To do better than this, we would need to show that for some reasonable fraction of $T$'s that there were no solutions.

\section*{Acknowledgements}

This work was done with the support of NSF and NDSEG graduate fellowships.  I would also like to thank Noam Elkies for suggesting the use of Heath-Brown's bound for points on a conic and Andrew Granville for the idea of adjusting the value of $\alpha$ in the lower bound to avoid short vectors in $L$.

\end{document}